\documentclass[reqno,11pt]{amsart}

\usepackage{amsmath,a4wide}
\usepackage{amsfonts}
\usepackage{verbatim}
\usepackage{mathtools}
\usepackage{amsthm}
\usepackage{indentfirst}
\usepackage{amsmath,amssymb}
\usepackage{amsthm}
\usepackage{fancyhdr}
\usepackage{mathrsfs}
\usepackage[hidelinks]{hyperref}
\usepackage{tikz}
\usepackage{subfig}
\usepackage{pgfplots}
\usepgfplotslibrary{polar}
\usepackage{caption}
\usepackage{graphicx}
\usepackage{enumerate}
\usepackage{mathrsfs}
\usepackage{amssymb}
\usepackage{amsthm}
\usepackage{amsmath,amsfonts,amssymb,esint}
\usepackage{graphics,color}
\usepackage{enumerate}
\usepackage{xfrac}
\usepackage{bbm}

\usepackage{marginnote}
\usepackage{xfrac}
\usepackage{mathtools}
\usepackage{hyperref}

\usepackage{comment}

\graphicspath{{Immagini//}}

\newcommand{\T}{\ensuremath{\mathbb{T}}}
\newcommand*{\R}{\ensuremath{\mathbb{R}}}

\newcommand*{\Z}{\ensuremath{\mathbb{Z}}}

\newcommand*{\supp}{\ensuremath{\mathrm{supp\,}}}

\renewcommand*{\div}{\ensuremath{\mathrm{div\,}}}

\newcommand*{\e}{\ensuremath{\varepsilon}}
\newcommand{\eps}{\varepsilon}

\newcommand*{\curl}{\ensuremath{\mathrm{curl\,}}}

\newcommand{\vertiii}[1]{{\left\vert\kern-0.25ex\left\vert\kern-0.25ex\left\vert #1 
    \right\vert\kern-0.25ex\right\vert\kern-0.25ex\right\vert}}

\newtheorem{theorem}{Theorem}[section]
\newtheorem{lemma}[theorem]{Lemma}
\newtheorem{proposition}[theorem]{Proposition}
\newtheorem{corollary}[theorem]{Corollary}

\begin{document}
\pagenumbering{arabic}
\title[]{Regularity in time of H\"older solutions of Euler and hypodissipative Navier--Stokes equations}

\author[Colombo]{Maria Colombo}
\address{EPFL SB, Station 8, 
CH-1015 Lausanne, Switzerland
}
\email{maria.colombo@epfl.ch}


\author[De Rosa]{Luigi De Rosa}
\address{EPFL SB, Station 8, 
CH-1015 Lausanne, Switzerland
}
\email{luigi.derosa@math.uzh.ch}

\begin{abstract} 
In this work we investigate some regularization properties of the incompressible Euler equations and of the fractional Navier-Stokes equations where the dissipative term is given by $(-\Delta)^\alpha$, for a suitable power $\alpha \in (0,\frac{1}{2})$ (the only meaningful range for this result). Assuming that the solution $u \in L^\infty _t(C^\theta_x)$ for some $\theta \in (0,1)$ we prove that $u \in C^\theta_{t,x}$, the pressure $p\in C^{2\theta-}_{t,x}$ and the kinetic energy $e \in C^{\frac{2\theta}{1-\theta}}_t$. This result was obtained  for the Euler equations  in \cite{Is2013} with completely different arguments and we believe that our proof, based on a regularization and a commutator estimate, gives a simpler insight on the result.
\end{abstract}

\maketitle

\textbf{Keywords:} Euler and Navier-Stokes equations, H\"older solutions, regularity of the pressure, time regularity, fractional dissipation.

\textbf{2010 Mathematics Subject Classification:} 35Q31 35A01 35D30.
\section{Introduction}
In the spatial periodic setting $\T^3=\R^3\setminus \Z^3$, we consider the partial differential equations
\begin{equation}\label{NSalpha}
\left\{\begin{array}{l}
\partial_t u+ \div (u \otimes u) +\nabla p + \nu (-\Delta)^\alpha u =0\\ 
\div u = 0\,
\end{array}\right.\qquad \mbox{in } [0,\infty) \times \T^3
\end{equation}
where $u: (0,T)\times \T^3  \rightarrow \R^3$ represents the velocity of an incompressible fluid, $p:(0,T)\times \T^3 \rightarrow \R$ is the hydrodynamic pressure, with the constraint $\int_{\T^3}p\, dx =0$ which guaranties its uniqueness, $\nu \in [0,\infty)$ is the viscosity of the fluid. When $\nu=0$ we have the Euler system, when $\nu>0$ and $\alpha=1$ the Navier-Stokes system, and for $\nu>0$ and $\alpha \in (0,1)$ we have the hypodissipative Navier-Stokes system.\\

We are concerned with a class of distributional solutions of the previous system, which exibit H\"older continuous spatial regularity.
H\"older solutions are of particular importance for the Euler equations in the context of hydrodynamic turbulence, starting from a celebrated prediction of Kolmogorov's theory \cite{K1941}: the velocity increments in turbulent flows should obey on average a universal scaling law corresponding to the H\"older exponent $\frac{1}{3}$
$$< |v(x+ \Delta x) - v(x)|^p >^{1/p} \leq \tilde C(p) |\Delta x|^{1/3}.$$

In 1949, Onsager  \cite{Ons1949} formulated a precise conjecture related to Kolmogorov's theory, which states that
weak solutions to the Euler equations with spatial H\"older regularity $u \in L^\infty((0,T), C^\theta(\T^3))$ 
\begin{itemize}
\item conserve energy if they are regular enough, namely if $\theta>\frac{1}{3}$,
\item may dissipate their kinetic energy if $\theta<\frac{1}{3}$.
\end{itemize}
 The first part of the conjecture has been proven in \cite{CoETi1994}; the construction of H\"older continuous anomalous solutions was initiated by De Lellis and Székelyhidi Jr. in \cite{DS2013} and was then performed in a series of papers \cite{DLS2012}, \cite{DLS2015}, \cite{BDS2016}  leading to the proof of the conjecture by Isett in \cite{Is2016}.

In the following we exploit a regularizing property of the Euler equations, namely that weak solutions to the Euler equations with spatial H\"older regularity $u \in L^\infty((0,T), C^\theta(\T^3))$ are in fact $\theta$ H\"older continuous also in time. Moreover, the associated pressure is almost $2\theta $ H\"older continuous in space-time and the corresponding kinetic energy profile is $\frac{2\theta}{1-\theta} $ H\"older continuous.
This property can be observed in all the dissipative solutions constructed to validate the Onsager conjecture and it was first proved by Isett \cite{Is2013}. In his work, this regularization is obtained as a consequence of the regularity for advective derivatives, and involves refined and technical estimates on the Paley-Littlehood decomposition of the solution. 

Our proof is based on completely different ideas, involving a regularization of the equation as in \cite{CoETi1994} as well as their commutator estimate. The method is quite flexible and indeed we perform it not only for Euler, but also for the fractional Navier-Stokes system, which was observed to exibit the existence of anomalous solutions for an appropriate range of $\alpha$ \cite{{CDD2017}}, \cite{DR2018}.

\begin{theorem}\label{t:main}
Let $\theta \in (0,1)$, $\alpha \in \big(0,\frac{1}{2}\big)$, $\nu \geq 0$, and let $(u, p)$ be a weak solution of  \eqref{NSalpha} such that $u\in L^\infty((0,T), C^\theta(\T^3))$. Then there exists $C_\nu>0$, depending only on $\nu$, such that 
\begin{equation}
\label{eqn:u-time}
\|u\|_{C^{\theta}([0,T]\times \T^3)} \leq C_{\nu}\big( \|u\|_{L^\infty((0,T),C^\theta(\T^3))}+ \|u\|^2_{L^\infty((0,T), C^\theta(\T^3))}\big)\,.
\end{equation}
Moreover there exists $C_\theta >0$ and for every $\varepsilon>0$ small a constant $ C_{\theta, \nu, \varepsilon}>0$ such that
\begin{itemize}
\item[$(i)$]if $\theta \in \big(0,\frac{1}{2}\big)$
\begin{align}
\|p\|_{ L^\infty((0,T), C^{2\theta}(\T^3))} &\leq C_\theta \|u\|^2_{L^\infty((0,T), C^\theta(\T^3))}\,, \label{ts:pspace}\\
\|p\|_{C^{2\theta- \eps}([0,T] \times \T^3)} &\leq C_{\theta, \nu, \varepsilon} \big( \|u\|^2_{L^\infty((0,T), C^\theta(\T^3))} +\|u\|^3_{L^\infty((0,T), C^\theta(\T^3))}  \big) \label{ts:ptime}\,;
\end{align}
\item[$(ii)$]if $\theta \in \big(\frac{1}{2},1\big)$
\begin{align}
\| p\|_{ L^\infty((0,T), C^{1,2\theta-1}(\T^3))} &\leq C_\theta \|u\|^2_{L^\infty((0,T), C^\theta(\T^3))}\,, \label{ts:pspace-thetahigh} \\
\| p\|_{C^{1,2\theta-1- \eps}([0,T] \times \T^3)} &\leq C_{\theta, \nu, \varepsilon} \sum_{m=2}^4\|u\|^m_{L^\infty((0,T), C^\theta(\T^3))}\label{ts:ptime-thetahigh}\,.
\end{align}
\end{itemize}
\end{theorem}
The assumption $\alpha<\frac{1}{2}$ is absolutely natural in this context: indeed, for $\alpha$ above this threshold any H\"older continuous solution to the $\alpha$-Navier-Stokes equation is in fact smooth by simple bootstrap arguments, based on the regularization of the ``fractional heat equation'' part of \eqref{NSalpha}, considering the nonlinearity and the pressure as a right-hand side.\\

In the following result, we consider the kinetic and total energies of the system \eqref{NSalpha}
\begin{equation}\label{energies}
e_u(t)=\frac{1}{2}\int_{\T^3}|u|^2(x,t)\,dx\,,\qquad
E_u(t)=e_u(t) + \nu \int_0^t\int_{\T^3}|(-\Delta)^{\sfrac{\alpha}{2}} u|^2(x,r)\,dxdr\,
\end{equation}
which coincide for the Euler equations, namely when $\nu =0$. We show that, instead of asking $u\in C^\theta$, a suitable spatial Sobolev (or Besov) regularity on the velocity $u$ is enough to guarantee H\"older regularity of the energies. This is obviously due to their "integral" nature.

\begin{theorem}\label{t:second}
Let $\theta \in (0,1)$, $\nu \geq 0$, $\alpha \in (0,\frac{1}{2})$ with the restriction that $\alpha<\theta$  if $\nu >0$. Let $u\in L^\infty((0,T), W^{\theta,3}(\T^3))$ or $u\in L^\infty((0,T), B^\theta_{3,\infty}(\T^3))$ be a weak solution of \eqref{NSalpha}. Then if $\nu =0$ (namely, for the Euler system) we have
$$|e_u(t)-e_u(s)|\leq C_{u,\theta}|t-s|^{\frac{2\theta}{1-\theta}}\,,$$
where $C_{u,\theta}=C_\theta\big(  [u]^2_{L^\infty((0,T),B^\theta_{3,\infty}(\T^3))}+[u]^3_{L^\infty((0,T),B^\theta_{3,\infty}(\T^3))} \big)$;
if $\nu>0$ (namely, for the hypodissipative Navier-Stokes system) we have 
\begin{equation}
\label{eqn:ts-energ}
|E_u(t)-E_u(s)|\leq C_{u,\theta,\alpha}|t-s|^\frac{2(\theta-\alpha)}{1-3\theta+2(\theta-\alpha)}\,,
\end{equation}
where $ C_{u,\theta,\alpha}=C_{\theta,\alpha}\big(    [u]^2_{L^\infty((0,T),B^\theta_{3,\infty}(\T^3))}+[u]^3_{L^\infty((0,T),B^\theta_{3,\infty}(\T^3))}\big)$.
\end{theorem}
Note that the previous theorem implies the energy conservation for \eqref{NSalpha}, in particular $e_u$ is conserved if $\theta>\frac{1}{3}$ and $E_u$ if $\theta >\max\{ \frac{1}{3}, \alpha\}$, since the H\"older exponents are greater than $1$.\\

To prove the time regularities, we look at a regularized version of \eqref{NSalpha} in the spirit of the proof of conservation of the energy for $\theta>\frac{1}{3}$ by  Constantin, E and Titi in \cite{CoETi1994}.

Let  $\rho \in C^ \infty_c (B_1(0))$ be a standard nonnegative  kernel such that  $\int_{B_1(0)} \rho(x) dx=1$. For any $ \delta>0 $ we define its rescaling $ \rho_\delta:=\delta^{-3} \rho(\frac{x}{\delta})$ and we consider the mollifications (in space) of $u$ and $p$
\[
u_\delta (t,x):=(u \ast \rho_\delta)(t,x)=\int_{B_\delta(x)}u(t,y) \rho_\delta(x-y) \,dy,
\qquad p_\delta (t,x):=(p \ast \rho_\delta)(t,x)\,.
\]
Thuis, mollifying equations \eqref{NSalpha} one gets
\begin{equation}\label{euler_moll}
\partial_t u_\delta+ \div (u_\delta \otimes u_\delta) +\nabla p_\delta +\nu (-\Delta)^\alpha u_\delta= \div R_\delta\,,
\end{equation}
where $R_\delta :=u_\delta \otimes u_\delta-(u \otimes u)_\delta$. It is easy to see that the energy identity for $u_\delta$ becomes
\begin{equation}\label{en_moll_eq}
\frac{d}{dt}\frac{1}{2}\int_{\T^3}|u_\delta|^2\,dx+\nu\int_{\T^3}|(-\Delta)^{\sfrac{\alpha}{2}}u_\delta|^2\,dx=\int_{\T^3}u_\delta\cdot \div R_\delta \,dx=-\int_{\T^3} R_\delta : \nabla u_\delta\,dx\,.
\end{equation}
Then we estimate the variation of $u$, $e_u$ and $E_u$ between two different times $s<t$ through $u_\delta$, $e_{u_\delta}$ and $E_{u_\delta}$ respectively, and we optimize the choice of $\delta$ in terms of $|t-s|$.

Regarding the pressure, taking the divergence of the first equation in \eqref{NSalpha} the pressure $p$ solves
\begin{equation}\label{laplace_p}
-\Delta p=\div \div (u\otimes u) \\
\end{equation}
and the solution is unique up to the renormalization $\int_{\T^3} p(t,x)\,dx =0$ for every $t \in [0,T]$.
By Schauder estimates
 we infer
\begin{equation}\label{p_ctheta}
\|p\|_{L^\infty((0,T),C^\theta(\T^3))}\leq C_\theta \|u\otimes u\|_{L^\infty((0,T),C^\theta(\T^3))}\leq C_\theta\|u\|^2_{L^\infty((0,T),C^\theta(\T^3))} \,.
\end{equation}

To improve the regularity as stated in \eqref{ts:pspace}, namely to show that $p$ is not only $\theta$-H\"older continuous but actually $2\theta$-H\"older continuous, we exploit
the quadratic structure of the right hand side $\div \div (u \otimes u)$, together with the solenoidal nature of the vector field $u$, to improve the regularity of the pressure. The space regularity of the pressure in $\T^3$ is then a direct consequence of Proposition \ref{t:2schau} and Lemma \ref{extension_lemma} below.
Finally, the time regularity of the pressure is obtained for $\theta<\frac{1}{2}$ by differentiating the equation \eqref{laplace_p}
$$
\Delta \partial_t p= \div \div \div (u \otimes u\otimes u) +\div \div \big(2 \nabla p \otimes u +\nu (-\Delta)^\alpha u \otimes u +\nu u\otimes (-\Delta)^\alpha u \big)\,,
 $$
and by exploiting again the structure of the right-hand side; 
in this case the presence of the fractional laplacian in the right-hand side introduces a technical difficulty to the analysis.

\section{Notations and technical preliminaries}
Our $3-$dimensional spatial domain $\Omega$ will be either an open set $\Omega \subseteq \R^3$ or $\Omega=\T^3$, thus considering vector fields $u:(0,T)\times  \Omega \rightarrow \R^3$ and a scalar field $p: (0,T)\times\Omega  \rightarrow \R$.We  denote by $ u^i$ the $i-$th component of the vector field $ u$, $\partial _i$ will be the derivative with respect to the $i-$th space variable. 

In what follows $\theta \in (0,1)$, $p\in[1,\infty)$,  $n \in \mathbb{N}$ and $\beta$ is a multi-index, $f$ is a (scalar or vector valued) function defined on $\Omega$. We introduce the usual (spatial) H\"older and $C^n(\Omega)$ norms
\begin{align*}
 \| f\|_{C^0(\Omega)}:=\sup_{x \in \Omega} \arrowvert f(x) \arrowvert\,,\quad
 [f]_{C^n(\Omega)}:=\sup_{|\beta|=n} \| D^\beta f\|_{C^0(\Omega)}\,,\quad
 \|f\|_{C^n(\Omega)}:=\| f \|_{C^0(\Omega)}+\sum_{j=1}^n[f]_{C^j(\Omega)}\,,
\end{align*}
\begin{align*}
   [f]_{C^\theta(\Omega)} := \sup_{x \neq y, \, x, y \in \Omega}\frac{|f(x)-  f(y)|}{|x-y|^\theta}\,,\quad \|f\|_{C^{n,\theta}(\Omega)} :=\|f\|_{C^n(\Omega)}+\sup_{|\beta|=n}[D^\beta f]_{C^\theta(\Omega)}.
\end{align*}
Denoting by $\|f\|_{L^p(\Omega)} := \big(\int_{\Omega} |f|^p(x)\, dx\big)^\frac{1}{p}$ the $L^p$-norm, we let the Sobolev and Besov norms
\begin{align*}
[f]_{W^{\theta,p}(\Omega)}  := \biggl(\int_{\Omega} \int_{\Omega} \frac{|f(x)-  f(y)|^p}{|x-y|^{\theta p + 3}}\, dx dy\biggl)^\frac{1}{p}, \qquad \|f\|_{W^{\theta,p}(\Omega)} := \|f\|_{L^p(\Omega)}+[f]_{W^{\theta,p}(\Omega)} \, .
 \,.
\end{align*}
\begin{align*}
[f]_{B^\theta_{p,\infty}(\Omega)}:= \sup_{y\in \Omega} \frac{\|f(\cdot + y)-f(\cdot)\|_{L^p(\Omega)}}{|y|^\theta}\, ,\qquad
\|f\|_{B^\theta_{p,\infty}(\Omega)}:=\|f\|_{L^p(\Omega)}+[f]_{B^\theta_{p,\infty}(\Omega)}\,.
\end{align*}
For the spaces defined above we have the trivial inclusions $C^\theta(\Omega)  \hookrightarrow B^\theta_{p,\infty}(\Omega)$ and $B^{\theta+\varepsilon}_{p,\infty}(\Omega) \hookrightarrow W^{\theta,p}(\Omega)$ for any $\varepsilon>0$, $p\in[1,\infty)$, $\theta \in(0,1)$.

\bigskip
%

\section{Proof of the main theorems}
To prove the space regularity of the pressure we exploit the explicit potential theoretic solution of the Laplace equation in $\R^3$. To this end, we denote by $\Phi(x):=\frac{1}{4 \pi |x|}$ the fundamental solution of the Laplace operator $-\Delta$, which enjoys the estimates
$\big|D^n \Phi(x)\big| \leq {C(n)}{|x|^{-1-n}} $ for all $n\in \mathbb{N}$. We recall that given $R \in C^{\theta}_c (\R^3; \R^{3 \times 3}) $ compactly supported, the potential theoretic solution of $-\Delta p = \div \div R$ is the only solution $p \in C^\theta(\R^3)$ which vanishes at infinity and it is given by the formula (with the Einstein summation convention)
\begin{equation}
\label{eqn:repr-p}
p(x) = \int_{B_{R_0}(x_0)} \partial^2_{ij} \Phi(x-y) \big( R^{ij}(y)- R^{ij}(x)\big)\, dy -R^{ij}(x) \int_{\partial B_{R_0}(x_0)} \partial_{i} \Phi(x-y) \nu_j(y)\, dy \,,
\end{equation}
where $B_{R_0}(x_0)$ is any ball containing the support of $R$ (and $R$ is thought to be extended to $0$ outside its support) and $\nu(y)$ is the normal to $B_{R_0}(x_0)$ at $y$. 
Notice that the first integrand is not singular around $x$ thanks to the H\"older regularity of $R$. Given any parameter $\lambda=\theta, \nu, \varepsilon$ we will explicitly write $C_{\lambda}$ to denote constants which depend only on $\lambda$.

\begin{proposition}\label{t:2schau}
Let $\beta,\gamma \in (0,1)$ and 
$v,w,z \in C^0(\R^3)$  be solenoidal vector fields compactly supported. If $p,q:\R^3 \rightarrow \R$ are the potential theoretic solutions of 
\begin{equation}\label{laplace_p_solospazio}
-\Delta p= \div \div (v\otimes w) \qquad \mbox{and} \qquad -\Delta q= \div \div  \div (v\otimes w\otimes z)\,,
\end{equation}
then there exist constants $C_{ \beta, \gamma}, C_\theta>0$ such that the following holds
\begin{itemize}
\item[(i)]if $\beta+ \gamma \in (0, 1)$ then $\|p\|_{C^{\beta+\gamma}(\R^3)} \leq C_{\beta, \gamma} \|v\|_{C^{\beta}(\R^3)} \|w\|_{C^{\gamma}(\R^3)}$,
\item[(ii)]if $\beta+ \gamma \in (1,2)$ then $\|p\|_{C^{1,\beta+\gamma-1}(\R^3)} \leq C_{\beta, \gamma}  \|v\|_{C^{\beta}(\R^3)} \|w\|_{C^{\gamma}(\R^3)}$,
\item[(iii)]if $\beta+ \gamma \in (1,2)$ then 
\begin{equation}\label{eqn:q}
\begin{split}
\| q\|_{C^{\beta+\gamma-1}(\R^3)} \leq &C_{\beta, \gamma} \|v\|_{C^0(\R^3)} \|w\|_{C^{\beta}(\R^3)} \|z\|_{C^{\gamma}(\R^3)}
\\&+C_{\beta, \gamma} \|v\|_{C^\beta(\R^3)} \big(\|w\|_{C^{0}(\R^3)} \|z\|_{C^{\gamma}(\R^3)}+ \|w\|_{C^{\gamma}(\R^3)} \|z\|_{C^{0}(\R^3)}\big).
\end{split}
\end{equation}
\end{itemize}
\end{proposition}
Taking $\beta= \gamma= \theta$ in the previous proposition and $v=w= u$, one obtains that the potential theoretic solutions $p$ and $q$ of $-\Delta p= \div \div (u\otimes u)$ and $-\Delta q= \div \div  \div (u\otimes u\otimes u)$ obey
\begin{equation}
\label{luigi}
\|p\|_{C^{2\theta}(\R^3)} \leq C_\theta \|u\|^2_{C^{\theta}(\R^3)} \qquad \mbox{if $\theta \in \big(0,\frac{1}{2}\big)$},
\end{equation}
$$\|p\|_{C^{1,2\theta-1}(\R^3)} \leq C_\theta \|u\|^2_{C^{\theta}(\R^3)}\quad \mbox{and}\quad \| q\|_{C^{2\theta-1}(\R^3)} \leq C_\theta\|u\|_{C^0(\R^3)} \|u\|^2_{C^{\theta}(\R^3)} \qquad \mbox{if $\theta \in \big(\frac{1}{2},1\big)$}.$$\\
However, the more general nature of Proposition~\ref{t:2schau} will be useful to deal with Theorem~\ref{t:main}(ii); in this case, we will take advantage not only of the structure of the equation for $p$, $\nabla p$ and $\partial_tp$, but also of the (previously showed) regularity of $u$ in time, and for this scope we will need Proposition~\ref{t:2schau} in its generality, including the nonsymmetric nature of \eqref{eqn:q} with respect to $v,w$ and $z$.
We do not expect \eqref{luigi} to hold for $\theta=\frac{1}{2}$ due to the usual loss in Schauder estimates in $C^n$-type, rather than H\"older, spaces; for $\theta=\frac{1}{2}$ the estimate reads $\|p\|_{C^{1-\varepsilon}(\R^3)}\leq C_\varepsilon \|u\|^2_{C^{\frac{1}{2}}(\R^3)}$
.
\begin{proof}
{\it (i)} We will prove that 
\begin{equation}
\label{ts:p-holder}[p]_{C^{\beta+\gamma}(\R^3)} \leq  C\|v\|_{C^{\beta}(\R^3)} \|w\|_{C^{\gamma}(\R^3)}.
\end{equation}
The estimate for $\|p\|_{C^0(\R^3)}$ (as well as the one for $ [p]_{C^{\min\{\beta, \gamma\}}(\R^3)}$) follows from the standard Schauder estimates. 
For any $x_1,x_2 \in \R^3$, we define $\bar x =\frac{x_1+x_2}{2}$ and $\lambda=|x_1-x_2|$. Since $\div  v=\partial_i  v^i=0= \div  w=\partial_i  w^i$, we observe that the equation for $p$ can be rewritten as
\begin{equation*}
\begin{split}
-\Delta p&= \partial_{ij} (v^i w^j) = \partial_{i} w^j \partial_{j} v^i = \partial_{i} (w^j-w^j(x_2)) \partial_{j} (v^i - v^i(x_1))
\\&= \partial_{ij} \big( (v^i - v^i(x_1)) (w^j-w^j(x_2)) \big) = \partial_{ij} \big( (v^i - v^i(x_1)) (w^j-w^j(x_2)) - v^i(x_1)w^j(x_2) \big) 
\end{split}
\end{equation*}
Since the function $ (v^i - v^i(x_1)) (w^j-w^j(x_2)) - v^i(x_1)w^j(x_2) $ is compactly supported in $B_1$, we conclude that the potential theoretic solution associated to it is the same as the potential theoretic solution associated to $v^i w^j$, namely by \eqref{eqn:repr-p} applied with $B_{R_0}(x_0)= B_{R_0}(\bar x)$ it is given by the representation formula
\begin{equation*}
\begin{split}
 p(x)&=\int_{B_{R_0}(\bar x)} \partial^2_{ij} \Phi(x-y) \big[ \big( v^i(y)- v^i(x_1)\big)\big(  w^j(y)-w^j(x_2)\big) - \big( v^i(x)-v^i(x_1)\big)\big(  w^j(x)-w^j(x_2)\big) \big]\, dy  
 \\&-\big[ \big( v^i(x)- v^i(x_1)\big)\big(  w^j(x)-w^j(x_2)\big) - v^i(x_1)w^j(x_2) \big]  \int_{\partial B_{R_0}(\bar x)} \partial_{i} \Phi(x-y) \nu_j(y)\, dy\,,
\end{split}
\end{equation*}
for every $x \in \R^3$. 
Through the isometry $y \to x_1+x_2-y$, using that $\partial_i \Phi$ is odd and $\nu(y) = -\nu (x_1+x_2-y)$, we observe that
\begin{equation*}
\begin{split}
\int_{\partial B_{R_0} (\bar x)}\!\!\!\!\!\!\!\! \partial_{i} \Phi(x_1-y) \nu_j(y)\, dy
&=\int_{\partial B_{R_0} (\bar x)} \!\!\!\!\!\!\!\!\partial_{i} \Phi(y-x_2) \nu_j(x_1+x_2-y)\, dy
=\int_{\partial B_{R_0} (\bar x)} \!\!\!\!\!\!\!\!\partial_{i} \Phi(x_2-y) \nu_j(y)\, dy\,.
\end{split}
\end{equation*}
Hence, we rewrite the incremental quotient as
\begin{align*}
 p(&x_1)- p( x_2)= \int_{B_{R_0}(\bar x)} \big[ \partial^2_{ij} \Phi(x_1-y) - \partial^2_{ij} \Phi(x_2-y) \big] \big(  v^i(y)- v^i(x_1)\big)\big(  w^j(y)- w^j(x_2)\big)\, dy 
\,.
\end{align*}
Splitting the contributions of $y \in B_\lambda(\bar x)$ and $y\in B^c_\lambda(\bar x)$,
\begin{align}\nonumber
 |p(&x_1)- p( x_2)| \leq \int_{(B_{R_0} 
 \setminus B_\lambda)(\bar x)} \!\!\!\!\!\!\!\!\big| \partial^2_{ij} \Phi(x_1-y) - \partial^2_{ij} \Phi(x_2-y) \big| \big|  v^i(y)- v^i(x_1)\big|\big|  w^j(y)- w^j(x_2)\big|\, dy 
  \\
 &+ \int_{B_{\lambda}(\bar x)} \big( |D^2\Phi(x_1-y)|+|D^2 \Phi(x_2-y)| \big) \big|  v^i(y)- v^i(x_1)\big|\big|  w^j(y)- w^j(x_2)\big|\, dy  = I + II\label{eqn:p-diff}
\end{align}
%
Using the decay of $|D^2 \Phi|$, for $k=1,2$ we estimate the second integral in the right-hand side of \eqref{eqn:p-diff} with 
\begin{align*}
II 
&\leq [v]_{C^{\beta}} [w]_{C^{\gamma}}   
\int_{B_\lambda(\bar x)} \big[ \lambda^{\gamma}\big|D^2 \Phi(x_1-y)\big| |x_1-y|^{\beta} + \lambda^{\beta}\big|D^2 \Phi(x_2-y)\big| |x_2-y|^{\gamma}\big]\,dy \\
& \leq C [v]_{C^{\beta}} [w]_{C^{\gamma}}  \Big( \int_{B_{2\lambda}( x_1)} \frac{\lambda^{\gamma}}{|x_1-y|^{3-\beta}}\,dy +
\int_{B_{2\lambda}( x_2)} \frac{\lambda^{\beta}}{|x_2-y|^{3-\gamma}}\,dy 
\Big) 
\leq C \lambda^{\beta+\gamma}[v]_{C^{\beta}} [w]_{C^{\gamma}} \,.
\end{align*}
By the decay of $|D^3 \Phi|$, in particular since for every point $\tilde x\in B_{\lambda/2}(\bar x)$ and for every $y\in B^c_\lambda(\bar x)$ we have $|\tilde x-y| \geq |\bar x-y| - |\tilde x-\bar x| \geq \frac{|\bar x-y|}{2}$ and 
$
\big|D^3 \Phi(\tilde x-y)\big| \leq \frac{C}{|\tilde x-y|^{4}} \leq \frac{C}{|\bar x-y|^{4}} 
$, we have
\begin{align*}
I &\leq \lambda [v]_{C^{\beta}} [w]_{C^{\gamma}}\int_{B^c_\lambda(\bar x)} \Big(\int_0^1 \big|D^3 \Phi(t x_1+ (1-t)x_2-y)\big| \,dt \Big)\, |x_1-y|^{\beta} |x_2-y|^{\gamma} \,dy \\
&\leq C \lambda [v]_{C^{\beta}} [w]_{C^{\gamma}} \int_{B^c_\lambda(\bar x)} \frac{ |x_1-y|^{\beta} |x_2-y|^{\gamma}}{|\bar x-y|^{4}} \,dy
\\&
\leq C \lambda[v]_{C^{\beta}} [w]_{C^{\gamma}} \int_{B^c_\lambda(\bar x)} \frac{1}{|\bar x-y|^{4-\beta-\gamma}} \,dy\leq  C \lambda^{\beta+\gamma}[v]_{C^{\beta}} [w]_{C^{\gamma}}\,.
\end{align*}
This concludes the proof of (i) (notice that in the last line we used that $\beta+\gamma<1$).\\

{\it (ii)} If $\beta+\gamma \in \big(\frac{1}{2},1\big)$ we have that for every partial derivative $\partial_k$ and for every given $x\in \R^3$
$$
-\Delta \partial_k p(y)= \partial^3_{ijk} (v^i (y)w^j(y))=\partial^2_{ij} \partial_k \big((v^i(y)-v^i(x))(w^j(y)-w^j(x))\big)\,.
$$
Since $\partial_k \big((v^i-v^i(x))(w^j-w^j(x))\big)$ is compactly supported we can use again the representation formula \eqref{eqn:repr-p} getting
\begin{align*}
\partial_k p(x)&=\int_{B_{R_0}(x_0)} \partial^2_{ij}\Phi(x-y) \partial_k \big((v^i(y)-v^i(x))(w^j(y)-w^j(x))\big)\,dy 
\,.
\end{align*}
Integrating by parts (this can be easily justified approximating $u$ with smooth functions) and letting $R_0 \rightarrow \infty$ we obtain
\begin{equation}\label{global_formula}
\partial_k p(x)=\int_{\R^3} \partial^3_{ijk}\Phi(x-y)(v^i(y)-v^i(x))(w^j(y)-w^j(x))\,dy\,.
\end{equation}
For every $x_1, x_2 \in \R^3$ we define $\overline x =\frac{x_1+x_2}{2}$, $\lambda=|x_1-x_2|$ and we write
\begin{align*}
\partial_k p(x_1)-&\partial_kp(x_2)= \int_{B_{\lambda}(\overline x)} \partial^3_{ijk}\Phi(x_1-y)(v^i(y)-v^i(x_1))(w^j(y)-w^j(x_1))\,dy \\
&-\int_{B_{\lambda}(\overline x)} \partial^3_{ijk}\Phi(x_2-y)(v^i(y)-v^i(x_2))(w^j(y)-w^j(x_2))\,dy \\
&+\int_{B^c_{\lambda}(\overline x)}\big(\partial^3_{ijk}\Phi(x_1-y)- \partial^3_{ijk}\Phi(x_2-y)\big)(v^i(y)-v^i(x_1))(w^j(y)-w^j(x_1))\,dy\\
&+\int_{B^c_{\lambda}(\overline x)} \partial^3_{ijk}\Phi(x_2-y) (v^i(y)-v^i(x_2))(w^j(x_2)-w^j(x_1))\,dy \\
&+\int_{B^c_{\lambda}(\overline x)} \partial^3_{ijk}\Phi(x_2-y) (v^i(x_2)-v^i(x_1))(w^j(y)-w^j(x_1))\,dy\,,
\end{align*}
and, arguing as in the proof for $\beta+\gamma <1$, it is easy to see that each of the above integrals is estimated by $C \lambda^{\beta+\gamma-1}[v]_{C^{\beta}} [w]_{C^{\gamma}}$ from which the estimate on $p$ in (ii) follows.\\

{\it (iii)} 
We note that for every choice of $x_1, x_2, x_0$ we can write $q=q^1+q^2$ where 
\begin{align*}
- \Delta q^1&= \partial^2_{ij} \partial_k \big( (v^i-v^i(x_1))(w^j-w^j(x_2))(z^k-z^k(x_0))\big)\,\\
-\Delta q^2&= \partial^3_{ijk} \big( (v^i(x_1) w^j z^k\big)+  \partial^3_{ijk} \big( (v^i w^j(x_2) z^k\big) +\partial^3_{ijk} \big( (v^i w^j z^k(x_0)\big)\,.
\end{align*}
Since the right hand side of the Poisson equation for $q^2$ has exactly the same structure of the one for $\partial_k p$ (the only difference are the constants but they do not play any role and they can be estimated by 
their respective $C^0$ norm) in the previous computations, we can infer that $q^2$ enjoys the estimate \eqref{eqn:q}.
For $q^1$ we can use the same formula as in \eqref{global_formula} choosing $x_0=x_m$ when we have to evaluate $q^1(x_m)$. Thus for $m=1,2$ we can write
$$
q^1(x_m)= \int_{\R^3} \partial^3_{ijk} \Phi (x_m-y) (v^i(y)-v^i(x_1))(w^j(y)-w^j(x_2))(z^k(y)-z^k(x_m))\,dy
$$
and again, letting $\lambda=|x_1-x_2|$, $\overline x=\frac{x_1+x_2}{2}$ and splitting the contributions in $B_\lambda(\overline x)$ and $ B_\lambda^c(\overline x)$ we write
\begin{align*}
&q^1(x_1)-q^1(x_2)=\int_{B_\lambda(\overline x)}\partial^3_{ijk} \Phi(x_1-y)(v^i(y)-v^i(x_1))(w^j(y)-w^j(x_2))(z^k(y)-z^k(x_1))\,dy \\
&-\int_{B_\lambda(\overline x)}\partial^3_{ijk} \Phi(x_2-y)(v^i(y)-v^i(x_1))(w^j(y)-w^j(x_2))(z^k(y)-z^k(x_2))\,dy \\
&+\int_{B^c_\lambda(\overline x)}\partial^3_{ijk} \Phi(x_1-y)(v^i(y)-v^i(x_1))(w^j(y)-w^j(x_2))(z^k(x_2)-z^k(x_1))\,dy \\
&+\int_{B^c_\lambda(\overline x)}\big( \partial^3_{ijk} \Phi(x_1-y)-\partial^3_{ijk}\Phi(x_2-y)\big) (v^i(y)-v^i(x_1))(w^j(y)-w^j(x_2))(z^k(y)-z^k(x_2))\,dy 
\,.
\end{align*}
We estimate each term in the same spirit as the previous computations to get
\begin{align*}
|q^1(x_1)&-q^1(x_2)| \leq C \|v\|_{C^\beta} \|w\|_{C^0} \|z\|_{C^\gamma} \int_{B_\lambda(\overline x)} \frac{1}{|x_1-y|^{4-\beta-\gamma}}\,dy 
+ C
\|v\|_{C^0} \|w\|_{C^\beta} \|z\|_{C^\gamma} \\
&
\times\Big( 
\int_{B_\lambda(\overline x)} \frac{1}{|x_2-y|^{4-\beta-\gamma}}\,dy 
+
\lambda^\gamma \int_{B_\lambda^c(\overline x)} \frac{1}{|\overline x -y|^{4-\beta} }\,dy 
+
\lambda\int_{B_\lambda^c(\overline x)} \frac{1}{|\overline x -y|^{5-\beta-\gamma}}\,dy 
\Big)
\\&
\leq C  \lambda^{\beta+\gamma-1} \big(\|v\|_{C^0(\R^3)} \|w\|_{C^{\beta}(\R^3)} \|z\|_{C^{\gamma}(\R^3)} + \|v\|_{C^\beta(\R^3)} \|w\|_{C^{0}(\R^3)} \|z\|_{C^{\gamma}(\R^3)}\big)\,, 
\end{align*}
which concludes the proof of (ii).
\end{proof}

In order to adapt the previous proposition to periodic solutions in $\R^3$ (thus without any decay at infinity) we will use the following lemma.
\begin{lemma}[Extension Lemma]\label{extension_lemma}
Let $\theta \in (0,1)$. For any $u \in C^\theta(\T^3)$ such that $\div u=0$, there exists a vector field $\tilde u \in C^\theta( \R^3)$ compactly supported in $B_{12}(0)$ and a positive constant $C_\theta>0$ such that $\div \tilde u=0$, $\tilde u \equiv u$ in $B_6(0)$ and
\begin{equation}\label{ext_cont}
\|\tilde u\|_ {C^\theta( \R^3)} \leq C_\theta \| u\|_ {C^\theta( \T^3)}\,.
\end{equation}
\end{lemma}
\begin{proof}
Assume for simplicity that $\int_{\T^3} u =0$ (the general case can be treated with a slight modification of the proof). Since $\div u=0$ on $\T^3$ then there exists a vector potential $A:\T^3 \rightarrow \R^3$ such that $u =\curl A$ and $-\Delta A= \curl  u$.
Moreover by Schauder estimates we have $\|A\|_{C^{1,\theta}( \T^3)} \leq C_\theta \|u\|_{C^{\theta}( \T^3)}$.  Now think $A$ to be defined periodically to the whole space $\R^3$. Choose a smooth cut-off function $0\leq\varphi \leq 1 $ such that $\supp\, \varphi \subset B_{12}$, $\varphi \equiv 1$ on $B_6$ and 
 $\|\varphi\|_{C^2}\leq C$
. Define $\tilde u := \curl \tilde A $ where $\tilde A := A \varphi_R\,$. Trivially $\div \tilde u=0$ and we also have the following  estimate
\[
\|\tilde v\|_ {C^\theta( \R^3)}\leq \|\tilde A\|_ {C^{1,\theta}( \R^3)}\leq C \|\varphi\|_ {C^{1,\theta}( \R^3)}\| A\|_ {C^{1,\theta}( \T^3)}\leq C_\theta \|u\|_{C^{\theta}( \T^3)}\,.
\]
Moreover $\tilde u$ satisfies  $\tilde u = \curl \tilde A = \curl A= u$ in $B_6(0)$.
\end{proof}
Note that the choice of $B_6(0)$ in the previous Lemma is to ensure that the cube (and thus the torus) $[-\pi,\pi]^3\subset B_6(0)$. Since we will work on functions $u,p$ that solve \eqref{laplace_p_solospazio} in $\T^3$, we can take the extension $\tilde{u}$ given by Lemma \ref{extension_lemma} and define $\tilde{p},\tilde q$ as 
$$
-\Delta \tilde p=\div \div (\tilde v \otimes \tilde w)\quad \text{in } \,\R^3\,.
$$
$$
-\Delta \tilde q=\div \div \div (\tilde v \otimes \tilde w\otimes \tilde z)\quad \text{in } \,\R^3\,.
$$
Thus we can write $p=p-\tilde p +\tilde p$ and $q=q-\tilde q+\tilde q$, where $\tilde p$ and $\tilde q$ satisfy Proposition \ref{t:2schau}, while $p-\tilde p$ and $q-\tilde q$ are harmonic in $B_6(0)$.
Thus we have  the following 
\begin{corollary}\label{to_torus}
If $v, w, z\in C^0(\T^3)$, then Proposition \ref{t:2schau} holds also for the (unique) zero-average solutions $p$ and $q$ of \eqref{laplace_p_solospazio} in $\T^3$.
\end{corollary}

\subsection{proof of Theorem \ref{t:main}} 
\subsubsection{Time regularity of $u$} To prove \eqref{eqn:u-time}, it is enough to show that $u$ is $\theta-$H\"older in time, uniformly in space. For any $s,t \in[0,T]$ we estimate
\begin{equation}\label{split_u}
|u(t,x)-u(s,x)|\leq |u(t,x)-u_\delta(t,x)|+|u_\delta(t,x)-u_\delta(s,x)|+|u_\delta(s,x)-u(s,x)|\,.
\end{equation}
Using \eqref{mollest2} we get
\[
|u(t,x)-u_\delta(t,x)|\leq C \delta^\theta \|u\|_{L^\infty((0,T),C^\theta(\T^3)))} \qquad \forall t\in [0,T]\, ,
\]
thus we are only left with the second term in the right hand side of \eqref{split_u}. Using the equation \eqref{euler_moll},  the estimates \eqref{p_ctheta} and \eqref{mollest3}, Theorem \ref{lapla.holder}, we have
\begin{align*}
|u_\delta(t,x)-u_\delta(s,x)|&\leq |t-s|\|\partial_t u_\delta\|_{L^\infty((0,T)\times\T^3)}\\
&\leq |t-s| \big(\|\div (u\otimes u)_\delta\|_{L^\infty((0,T)\times\T^3)} +\|\nabla p_\delta\|_{L^\infty((0,T)\times\T^3)} + \nu\|(-\Delta)^\alpha u_\delta\|_{L^\infty((0,T)\times\T^3)} \big)\\
&\leq C |t-s| \big( \delta^{\theta-1}\|u\|^2_{L^\infty((0,T),C^\theta(\T^3))}+  \nu \| u_\delta \|_{L^\infty((0,T),C^{2\alpha+\eps}(\T^3))}\big) \,,
\end{align*}
Since $\alpha \in \big(0,\frac{1}{2}\big)$ we can choose $\varepsilon$ such that $2\alpha+\varepsilon <1$, getting
\begin{align*}
|u_\delta(t,x)-u_\delta(s,x)|\leq C  |t-s|  \delta^{\theta-1}\big( \|u\|^2_{L^\infty((0,T),C^\theta(\T^3))}+\nu  \| u \|_{L^\infty((0,T),C^\theta(\T^3))}\big) \,,
\end{align*}
Finally we choose $\delta=|t-s|$, from which the claim follows.

\subsubsection{Space regularity for $p$, for $\theta \in (0,1)$} Estimates \eqref{ts:pspace} and \eqref{ts:pspace-thetahigh} follow from Corollary \ref{to_torus}.

\subsubsection{Time regularity for $p$, for $\theta<\frac{1}{2}$} For any $s,t \in[0,T]$, such that $|t-s|=\delta<1$ we estimate via the triangular inequality and thanks to the space regularity of $p$ and  \eqref{mollest2} 
\begin{equation}\label{split_p}
\begin{split}
|p(t,x)-p(s,x)|&\leq2 \sup_{t\in [0,T]} |p(t,x)-p_\delta(t,x)|+|p_\delta(t,x)-p_\delta(s,x)|
\\&
\leq  C \delta^{2\theta} \|p\|_{L^\infty((0,T),C^{2\theta}(\T^3))} +|p_\delta(t,x)-p_\delta(s,x)| \\&
\leq C \delta^{2\theta} \|u\|^2_{L^\infty((0,T), C^\theta(\T^3))} +|p_\delta(t,x)-p_\delta(s,x)|
\,.
\end{split}
\end{equation}
To estimate the last term we consider the equation solved by $p_\delta$
$$ -\Delta p_\delta= \div \div ( ( u\otimes u)_\delta ) = \div \div(   u_\delta \otimes u_\delta-R_\delta )$$
and hence the one for $p_\delta(t,\cdot)-p_\delta(s,\cdot)$
\begin{equation}\label{ugly_p_moll}
\begin{split}
-\Delta ( &p_\delta(t,x) -p_\delta(s,x) ) = \div \div( R_\delta(s,x) - R_\delta(t,x) +  u_\delta(t,x) \otimes u_\delta(t,x) - u_\delta(s,x) \otimes u_\delta(s,x) )
\\&= \div \div \Big[ R_\delta(s,x) - R_\delta(t,x) + \int_s^t \Big( \frac{d}{dr}u_\delta(r,x) \otimes u_\delta(r,x)+ u_\delta(r,x) \otimes  \frac{d}{dr} u_\delta(r,x) \Big) \, dr \Big]
\\&= \div \div\Big[ R_\delta(s,x) - R_\delta(t,x) +
 \int_s^t \Big( (\div(u_\delta \otimes u_\delta)- \nabla p_\delta- \div R_\delta-\nu(-\Delta)^\alpha u_\delta) \otimes u_\delta
 \\&\qquad+ u_\delta \otimes  (\div(u_\delta \otimes u_\delta)- \nabla p_\delta- \div R_\delta-\nu(-\Delta)^\alpha u_\delta)  \Big) \, dr \Big]\,.
\end{split}
\end{equation}
Defining the commutator 
\begin{equation}\label{T_commutator}
T^\alpha (u_\delta) :=(-\Delta)^\alpha (u_\delta \otimes u_\delta)-(-\Delta)^\alpha u_\delta \otimes u_\delta - u_\delta\otimes (-\Delta)^\alpha u_\delta\,,
\end{equation}
and denoting by $p^1_{s,t}$, $p^2_{s,t}$, $p^3_{s,t}, p^4_{s,t}, p^5_{s,t}$ the unique $0$-average solutions of
$$-\Delta p^1_{s,t} = \div \div( R_\delta(s,x) - R_\delta(t,x))\,,$$
$$\Delta p^2_{s,t} = \int_s^t \div \div( (\div R_\delta+ \nabla p_\delta) \otimes u_\delta +u_\delta
 \otimes (\div R_\delta+ \nabla p_\delta) ) \, dr\,,$$
$$-\Delta p^3_{s,t} = \int_s^t \div \div( \div(u_\delta \otimes u_\delta) \otimes u_\delta +u_\delta
 \otimes \div(u_\delta \otimes u_\delta) ) \, dr\,,$$
 $$
 -\Delta p^4_{s,t}=\nu \int_s^t \div \div T^\alpha (u_\delta)\,dr\,,
 $$
 $$
 \Delta p^5_{s,t}=\nu \int_s^t \div \div (-\Delta)^\alpha (u_\delta\otimes u_\delta)\,dr\,,
 $$
we have that
$$p_\delta(t,x) -p_\delta(s,x) = p^1_{s,t}+ p^2_{s,t} +p^3_{s,t}+p^4_{s,t}+p^5_{s,t}\,.$$
By Schauder estimates, estimating $R_\delta$ by \eqref{mollest4} and $p_\delta$ by \eqref{mollest2} and thanks to the space regularity of $p$ proved above in  \eqref{ts:pspace}, $p^1_{s,t}$ and $p^2_{s,t}$ enjoy the estimate
$$\|p^1_{s,t} \|_{L^\infty(\T^3)} \leq \|p^1_{s,t} \|_{C^\eps(\T^3)} \leq C\big( \|R_\delta(t,\cdot)\|_{C^\eps(\T^3)} + \|R_\delta(s,\cdot)\|_{C^\eps(\T^3)}\big) \leq C \delta^{2 \theta-\eps} \|u\|^2_{L^\infty((0,T),C^\theta(\T^3))},$$
\begin{equation*}
\begin{split}
\|p^2_{s,t} \|_{L^\infty(\T^3)}& \leq \|p^2_{s,t} \|_{C^\eps(\T^3)} \leq C|t-s| \| (\div R_\delta+ \nabla p_\delta) \otimes u_\delta \|_{L^\infty((0,T),C^\eps(\T^3))}
 \\&
 \leq C |t-s| \Big(\| R_\delta\|_{L^\infty((0,T),C^{1,\eps}(\T^3))} + \| p_\delta\|_{L^\infty((0,T),C^{1,\eps}(\T^3))} \Big)\| u_\delta \|_{L^\infty((0,T),C^\eps(\T^3))} \\&
 \leq C |t-s|  \delta^{2 \theta-\eps-1} \|u\|^3_{L^\infty((0,T),C^\theta(\T^3))} \,.
\end{split}
\end{equation*}
Note that $p^3_{s,t}$ is the integral in time of $\Delta^{-1}\div \div \div(u_\delta \otimes u_\delta \otimes u_\delta ) $, which from Corollary~\ref{to_torus} and \eqref{mollest1} are controlled by 
\begin{equation}
\begin{split}
\|p^3_{s,t} \|_{L^\infty(\T^3)} \leq \|p^3_{s,t} \|_{C^\eps(\T^3)} &\leq C |t-s| \|u_\delta\|_{L^\infty((0,T)\times \T^3)}\|u_\delta\|^2_{L^\infty((0,T),C^{\frac{1+\varepsilon}{2}}(\T^3))}
 \\&
 \leq C |t-s|  \|u\|_{L^\infty((0,T)\times \T^3)} \big( \delta^{\theta-\frac{1+\varepsilon}{2}}\|u\|_{L^\infty((0,T),C^\theta(\T^3))}\big)^2 \\&
 \leq C |t-s| \delta^{2\theta-1-\epsilon}\|u\|^3_{L^\infty((0,T),C^\theta(\T^3))}\,.
\end{split}
\end{equation}
Choosing $\varepsilon$ such that $\frac{\varepsilon}{2}<\alpha$, by Schauder estimates and \eqref{est_comm_1} we have 
$$
\|p^4_{s,t}\|_{L^\infty(\T^3)}\leq \|p^4_{s,t}\|_{C^\varepsilon(\T^3)}\leq  C|t-s|\|T^\alpha (u_\delta)\|_{L^\infty((0,T),C^\varepsilon(\T^3))}\leq C |t-s| \|u_\delta\|^2_{{L^\infty((0,T),C^{\alpha+\frac{\varepsilon}{2}}(\T^3))}}\,.
$$
To estimate $p^5_{s,t}$ we note that every solution of $\Delta q=\div \div (u_\delta \otimes u_\delta)$ enjoy the estimate (by Proposition \ref{to_torus})
$
\|q\|_{L^\infty((0,T),C^{2\alpha+\varepsilon}(\T^3))}\leq C \|u_\delta\|^2_{L^\infty((0,T),C^{\alpha+\frac{\varepsilon}{2}}(\T^3))}\,,
$
and since $p^5_{s,t}=\int_s^t (-\Delta)^\alpha q\,dr$, by Theorem \ref{lapla.holder} 
we infer
$$
\|p^5_{s,t}\|_{L^\infty(\T^3)}\leq C |t-s| \|q\|_{L^\infty((0,T),C^{2\alpha+\varepsilon}(\T^3))}\leq C |t-s| \|u_\delta\|^2_{{L^\infty((0,T),C^{\alpha+\sfrac{\varepsilon}{2}}(\T^3))}}\,. 
$$
In the case $\alpha<\theta$, if $\varepsilon$ is sufficiently small, we have (since $\delta<1$)
$$
\|u_\delta\|^2_{{L^\infty((0,T),C^{\alpha+\sfrac{\varepsilon}{2}}(\T^3)}}\leq \|u\|^2_{{L^\infty((0,T),C^{\theta}(\T^3))}}\leq \delta^{2\theta-1}\|u\|^2_{{L^\infty((0,T),C^{\theta}(\T^3))}}\,,
$$
while, if $\alpha\geq \theta$, using \eqref{mollest1} we have
$$
\|u_\delta\|^2_{{L^\infty((0,T),C^{\alpha+\sfrac{\varepsilon}{2}}(\T^3))}}\leq \delta^{2(\theta-\alpha)-\varepsilon} \|u\|^2_{{L^\infty((0,T),C^{\theta}(\T^3))}}\leq \delta^{2\theta-1}\|u\|^2_{{L^\infty((0,T),C^{\theta}(\T^3))}}\,.
$$
Thus we deduce
$$
\|p^4_{s,t}\|_{L^\infty(\T^3)}+ \|p^5_{s,t}\|_{L^\infty(\T^3)}\leq C|t-s| \delta^{2\theta-1}\|u\|^2_{{L^\infty((0,T),C^{\theta}(\T^3))}}\,.
$$
Since $\delta=|t-s|$ we conclude that  $|p_\delta(t,x)-p_\delta(s,x)| \leq C|t-s|^{2\theta-\varepsilon}$, so that \eqref{ts:ptime} holds true.

\subsubsection{Time regularity for $\nabla p$, for $\theta>\frac{1}{2}$}
By the equation solved by $p$, for $0<s<t$ we have
\begin{align*}
-\Delta( p(t) - p(s) )&= \partial_{ij}^2 (u^i(t) u^j(t)-u^i(s) u^j(s)) 
\\&= \partial_{ij}^2 \big((u^i(t)- u^i(s)) u^j(t)+u^i(s) (u^j(t)-u^j(s))\big) 
\end{align*}
By Corollary \ref{to_torus} we can apply Proposition~\ref{t:2schau}(ii) with $\beta= 1-\theta+\e$ and $\gamma= \theta$ to obtain
$$\| \nabla p(t)- \nabla p(s)\|_{C^\e(\T^3)} \leq C \|u(t)-u(s)\|_{C^{1-\theta+\e}(\T^3)}\|u\|_{{L^\infty((0,T),C^{\theta}(\T^3))}}.
$$
Interpolating the $C^{1-\theta+\e}$-norm between $C^\theta$ and $C^0$ (since $\theta>\frac12$) and by the time regularity of $u$ in \eqref{eqn:u-time}, we have 
\begin{equation*}
\begin{split}
\|u(t)-u(s)\|_{C^{1-\theta+\e}(\T^3)}&\leq 
\|u(t)-u(s)\|_{C^{\theta}(\T^3)}^{\frac{1-\theta+\e}{\theta}}
\|u(t)-u(s)\|_{C^{0}(\T^3)}^{\frac{2\theta-1-\e}{\theta}}
\leq
\|u\|_{C^{\theta}([0,T]\times \T^3)}|t-s|^{2\theta-1-\e}
\\&\leq \big( \|u\|_{L^\infty((0,T), C^\theta(\T^3))} +\|u\|^2_{L^\infty((0,T), C^\theta(\T^3))}  \big) |t-s|^{2\theta-1-\e},
\end{split}
\end{equation*}
which proves that  for any $x\in \T^3$
\begin{equation}\label{ciao}
\big| \nabla p(t,x)-\nabla p(s,x) \big| \leq C |t-s|^{2\theta-1-\eps}
\end{equation}

\subsubsection{Space regularity for $\partial_t p$, for $\theta>\frac{1}{2}$} With the previous arguments, $p \in C^{0,1}([0,T] \times \T^3)$. Hence $\partial_t p \in L^\infty$ and we can look at the equation solved distributionally by it, obtained by differentiating in time \eqref{laplace_p}
$$-\Delta \partial_t p = \div \div \partial_t ( u \otimes u)\,.$$
Note that, defining $T^\alpha (u_\delta)$ as in \eqref{T_commutator}, for every $\delta >0$ we have
\begin{align*}
\partial_t ( u_\delta \otimes u_\delta )&=\partial_t u_\delta \otimes u_\delta +u_\delta \otimes \partial_t u_\delta \\&=- \div (u_\delta \otimes u_\delta \otimes u_\delta) + \div R_\delta \otimes u_\delta +u_\delta \otimes \div R_\delta  \\
&\quad \,-\nabla p_\delta \otimes u_\delta -u_\delta \otimes \nabla p_\delta+\nu T^\alpha ( u_\delta)-\nu(-\Delta)^\alpha(u_\delta \otimes u_\delta)\,
\end{align*}
and, since by \eqref{mollest4} $\div R_\delta \rightarrow 0$ uniformly and by Proposition \ref{commutator_alpha} $T^\alpha (u_\delta) \rightarrow T^\alpha (u)$ uniformly, we have that $\partial_t p$ solves distributionally
\begin{equation}
\label{eqn:partial-t-p}
\Delta \partial_t p= \div \div \div (u \otimes u\otimes u) +\div \div \big(2 \nabla p \otimes u -\nu T^\alpha (u) +\nu (-\Delta)^\alpha(u\otimes u)\big)\,.
\end{equation}
Hence we can write $\partial_t p = q^1 + q^2+q^3+q^4$, where 
$$\Delta q^1 = \div \div \div(u\otimes u \otimes u ) \qquad \qquad \Delta q^2 =2 \,\div \div (\nabla p \otimes u  )  $$
$$\Delta q^3 =\nu\, \div \div T^\alpha (u) 
 \qquad \qquad
 -\Delta q^4 =\nu \,\div \div(-\Delta)^\alpha(u\otimes u)$$
In turn by the estimate on $q= q^1$ from Corollary \ref{to_torus}
$$\|q^1(t)\|_{C^{2\theta-1}(\T^3)} \leq C \|u (t) \|_{C^\theta(\T^3)}^3$$
and by Schauder estimates, together with the regularity of $p$, we have
$$\|q^2(t)\|_{C^{2\theta-1}(\T^3)} \leq C \|\nabla p \otimes u\|_{C^{2\theta-1}(\T^3)}  \leq  C \|\nabla p (t) \|_{C^{2\theta-1}(\T^3)} \|u (t) \|_{C^\theta(\T^3)} \leq C \|u (t) \|_{C^\theta(\T^3)}^3\,.$$
Applying \eqref{est_comm_1} and \eqref{est_comm_2} with $\beta=2(\theta-\alpha)$ (choosing $\varepsilon$ small enough such that $\theta<1-\varepsilon$) and by Schauder estimates we deduce
$$
\|q^3(t)\|_{C^{2\theta-1}(\T^3)}\leq C\|T^\alpha (u)(t)\|_{C^{2\theta-1}(\T^3)}\leq C \|u(t)\|^2_{C^\theta(\T^3)}\leq C \|u(t)\|^2_{C^\theta(\T^3)}\,.
$$
Notice that $q^4=\nu(-\Delta)^\alpha p$, thus by \eqref{ts:pspace-thetahigh} we have
$$
\|q^4(t)\|_{C^{2\theta-1}(\T^3)}\leq C \|p(t)\|_{C^{1,2\theta-1}(\T^3)}\leq C \|u(t)\|^2_{C^\theta(\T^3)}\,.
$$
\subsubsection{Time regularity for $\partial_t p$, for $\theta>\frac{1}{2}$}

For any $0<s<t$, by the equation for $\partial_t p$ in \eqref{eqn:partial-t-p}, we have that $\partial_t p(t,x) -\partial_t p(s,x) = p^1_{s,t}+ p^2_{s,t} +p^3_{s,t}
$
where $p^1_{s,t}, p^2_{s,t}$, and $p^3_{s,t}$ are the unique $0$-average solutions in $\T^3$ of
\begin{equation*}
\begin{split}
\Delta p^1_{s,t} &= \div \div \div (u(t) \otimes u(t)\otimes u(t)-u(s) \otimes u(s)\otimes u(s))
\\&= \div \div \div \big(( u(t)-u(s)) \otimes u(t)\otimes u(t)+ u(s) \otimes( u(t)-u(s))\otimes u(t)\\
&\quad \,+ u(s) \otimes u(s)\otimes (u(t)- u(s))\big)
\end{split}
\end{equation*}
\begin{equation*}
\begin{split}
\Delta p^2_{s,t} =\, \div \div \big(2\,\nabla p(t) \otimes u(t)- 2\,\nabla p(s) \otimes u(s)  -\nu T^\alpha (u(t)) +\nu T^\alpha (u(s)) \big)
\end{split}
\end{equation*}
$$\Delta p^3_{s,t} =\nu (-\Delta)^\alpha \div \div \big(u(t)\otimes u(t)-u(s)\otimes u(s)\big)\,.$$
To estimate $p^1_{s,t}$, for any $\e$ small we apply Corollary \ref{to_torus}, with particular reference to Proposition~\ref{t:2schau}(iii), with $\beta= 1-\theta+\e$ and $\gamma=\theta$, in such a way that the factor $u(t)-u(s)$ gets each time only the $C^{1-\theta+\e}$ norm and not the $C^\theta$ norm. We obtain that
$$\|p^1_{s,t}\|_{L^\infty} \leq \|p^1_{s,t}\|_{C^\e} \leq 
C \|u(t)-u(s)\|_{C^{1-\theta+\e}(\T^3)} \|u\|^2_{L^\infty((0,T),C^{\theta}(\T^3))}.
$$
Next, we interpolate the $C^{1-\theta+\e}$ norm between $C^0$ and $C^\theta$ and finally we use the $C^\theta$ regularity in time of $u$ to obtain that
\begin{equation*}
\begin{split}
\|u(t)-u(s)\|_{C^{1-\theta+\e}(\T^3)} &\leq C\|u(t)-u(s)\|_{C^{\theta}(\T^3)}^{\frac{1-\theta+\e}\theta}\|u(t)-u(s)\|_{C^{0}(\T^3)}^{{\frac{2\theta-1-\e}\theta}}
\\& \leq C\|u\|_{L^\infty((0,T),C^{\theta}(\T^3))}^{\frac{1-\theta+\e}\theta}|t-s|^{2\theta-1-\e}\|u\|_{C^{\theta}([0,T]\times \T^3)}^{{\frac{2\theta-1-\e}\theta}}\\
& \leq C |t-s|^{2\theta-1-\e}\|u\|_{C^{\theta}([0,T]\times \T^3)}.
\end{split}
\end{equation*}
Notice that 
$$
T^\alpha(u(t))-T^\alpha(u(s))=T^\alpha \big( u(t)-u(s),u(t)\big)+T^\alpha\big( u(s), u(t)-u(s) \big)\,.
$$
Now if $2\alpha>\theta$ we use part (i) of Theorem \ref{commutator_alpha} with $k_1=2\alpha-\theta+\frac{\varepsilon}{2}$, $k_2=\theta-\frac{\varepsilon}{2}$, $\beta=\varepsilon$ and we get
$$
\|T^\alpha \big( u(t)-u(s),u(t)\big)\|_{C^\varepsilon(\T^3)}\leq C\| u(t)-u(s)\|_{C^{2\alpha-\theta+\varepsilon}(\T^3)}\|u(t)\|_{C^\theta(\T^3)}\,,
$$
while if $2\alpha\leq \theta$ we choose $k_1=\varepsilon$, $k_2=2\alpha-\varepsilon$, $\beta=\varepsilon$, getting
\begin{align*}
\|T^\alpha \big( u(t)-u(s),u(t)\big)\|_{C^\varepsilon(\T^3)}&\leq C \| u(t)-u(s)\|_{C^{3\sfrac{\varepsilon}{2}}(\T^3)}\|u(t)\|_{C^{2\alpha-\sfrac{\varepsilon}{2}}(\T^3)}\\
&\leq C \| u(t)-u(s)\|_{C^{2\varepsilon}(\T^3)}\|u(t)\|_{C^{\theta}(\T^3)}\,.
\end{align*}
Using again the H\"older regularity of $u$ in time, we obtain
\begin{align*}
\|T^\alpha \big( u(t)-u(s),u(t)\big)\|_{C^{\sfrac{\varepsilon}{2}}(\T^3)}\leq C |t-s|^{2\theta-1-\varepsilon}\|u\|_{C^\theta([0,T]\times \T^3)}\|u\|_{L^\infty((0,T),C^\theta(\T^3))}\,.
\end{align*}
Summarizing we achieved
\begin{equation}\label{comm_tempo}
\| T^\alpha(u(t))-T^\alpha(u(s)) \|_{C^{\sfrac{\varepsilon}{2}}(\T^3)}\leq C |t-s|^{2\theta-1-\varepsilon} \|u\|_{C^\theta([0,T]\times \T^3)}\|u\|_{L^\infty((0,T),C^\theta(\T^3))}\,.
\end{equation}
Moreover by interpolation we estimate
\begin{align}\label{press_tempo}
 \|\nabla p(t) \otimes u(t)- \nabla p(s) \otimes u(s)  \|_{C^{\sfrac{\eps}{2}}(\T^3)}&\leq\|\nabla p(t) \otimes u(t)- \nabla p(s) \otimes u(s)  \|_{C^0(\T^3)}^{1-\frac{\varepsilon}{2(2\theta-1-\sfrac{\e}{2})}} \nonumber \\ & \,\quad \, \|\nabla p(t) \otimes u(t)
 - \nabla p(s) \otimes u(s)  \|_{C^{2\theta-1-\sfrac{\e}{2}}(\T^3)}^{\frac{\varepsilon}{2(2\theta-1-\sfrac{\e}{2})}}\nonumber \\
 &\leq C |t-s|^{2\theta-1-\varepsilon} \|\nabla p\|_{C^{2\theta-1-\sfrac{\varepsilon}{2}}([0,T]\times\T^3)} \|u\|_{C^\theta([0,T]\times\T^3)}\,.
\end{align}
By Schauder estimates, together with \eqref{comm_tempo} and \eqref{press_tempo}, we conclude
$$ \|p^2_{s,t} \|_{C^{\sfrac{\eps}{2}}(\T^3)} \leq C|t-s|^{2\theta-1-\varepsilon}\,.
$$
We note that  $p^3_{s,t}=-\nu (-\Delta)^\alpha (p(t)-p(s))$, thus by Theorem \ref{lapla.holder} and  \eqref{ciao} we have
$$
\|p^3_{s,t}\|_{C^\varepsilon(\T^3)}\leq C \|p(t)-p(s)\|_{C^1(\T^3)}\leq C|t-s|^{2\theta-1-\varepsilon}\,.
$$

$$
[\partial_t p(x)]_{C^{2\theta-1-\varepsilon}([0,T])}\leq C \|u\|^3_{L^\infty((0,T),C^\theta(\T^3))}\,.
$$
\subsection{proof of Theorem \ref{t:second}} 
\subsubsection{The case $\nu=0$ (Euler)}Let $s,t\in [0,T]$.  We wish to find a proper estimate for $|e_u(t)-e_u(s)|$. To do this we split it in three terms as follows
\begin{equation}\label{split_energy}
|e_u(t)-e_u(s)|\leq |e_u(t)-e_{u_\delta}(t)|+|e_{u_\delta}(t)-e_{u_\delta}(s)|+|e_{u_\delta}(s)-e_u(s)|\,,
\end{equation}
for some parameter $\delta>0$ that will be fixed at the end of the proof. Assume that $u\in L^\infty((0,T),B^\theta_{3,\infty}(\T^3))$ (the case $u\in L^\infty((0,T),W^{\theta,3}(\T^3))$ is analogous). Using \eqref{mean_moll} and \eqref{est4} with $p=\frac{3}{2}$ we can estimate
\[
|e_u(t)-e_{u_\delta}(t)|\leq \frac{1}{2} \int_{\T^3} \big| (|u|^2)_\delta-|u_\delta|^2\big| (x,t)\,dx \leq C \delta^{2\theta} [u]^2_{L^\infty((0,T),B^\theta_{3,\infty}(\T^3))}\,.
\]
We are now left with the second therm in the right hand side of \eqref{split_energy}. By \eqref{en_moll_eq} we get
\[
|e_{u_\delta}(t)-e_{u_\delta}(s)|\leq |t-s| \bigg\| \frac{d e_\delta}{dt}\bigg\|_{L^\infty(0,T)}\leq C |t-s| \|R_\delta\|_{L^\infty((0,T),L^{\sfrac{3}{2}}(\T^3))}\|\nabla u_\delta\|_{L^\infty((0,T),L^3(\T^3))}\,,
\]
and using \eqref{est1} and \eqref{est4} we obtain
\[
|e_{u_\delta}(t)-e_{u_\delta}(s)|\leq C|t-s| \delta^{3\theta -1}[u]^3_{L^\infty((0,T),B^\theta_{3,\infty}(\T^3))}\,.
\]
Thus we have achieved
\[
|e_u(t)-e_u(s)|\leq C \big( [u]^2_{L^\infty((0,T),B^\theta_{3,\infty}(\T^3))}+[u]^3_{L^\infty((0,T),B^\theta_{3,\infty}(\T^3))} \big)\big(\delta^{2\theta} +|t-s| \delta^{3\theta -1}\big)\,,
\]
from which choosing $\delta=|t-s|^\frac{1}{1-\theta}$ we conclude
\[
|e_u(t)-e_u(s)|\leq C\big(    [u]^2_{L^\infty((0,T),B^\theta_{3,\infty}(\T^3))}+[u]^3_{L^\infty((0,T),B^\theta_{3,\infty}(\T^3))} \big)|t-s|^\frac{2\theta}{1-\theta}\,.
\]
\subsubsection{The case $\nu>0$ (Hypodissipative Navier-Stokes)} We assume that $u\in L^\infty((0,T),B^\theta_{3,\infty}(\T^3))$ and we spilt
\begin{equation}\label{split_energy}
|E_u(t)-E_u(s)|\leq |E_u(t)-E_{u_\delta}(t)|+|E_{u_\delta}(t)-E_{u_\delta}(s)|+|E_{u_\delta}(s)-E_u(s)|\,,
\end{equation}
Using \eqref{mean_moll} and \eqref{est4} with $p=\frac{3}{2}$ we can estimate
\begin{align*}
|E_u(t)-E_{u_\delta}(t)| &\leq \frac{1}{2} \int_{\T^3} \big| (|u|^2)_\delta-|u_\delta|^2\big| (x,t)\,dx+\nu\int_0^t  \int_{\T^3} \big| (|(-\Delta)^{\sfrac{\alpha}{2}} u|^2)_\delta-|(-\Delta)^{\sfrac{\alpha}{2}}u_\delta|^2\big| (x,r)\,dxdr
\end{align*}
Using  \eqref{est4} with $p=\frac{3}{2}$ and with $p=1$ we have respectively
$$
\frac{1}{2} \int_{\T^3} \big| (|u|^2)_\delta-|u_\delta|^2\big| (x,t)\,dx \leq C \delta^{2\theta}[u]^2_{L^\infty((0,T),B^\theta_{3,\infty}(\T^3))}\,,
$$
$$
\int_0^t  \int_{\T^3} \big| (|(-\Delta)^{\sfrac{\alpha}{2}} u|^2)_\delta-|(-\Delta)^{\sfrac{\alpha}{2}}u_\delta|^2\big| (x,r)\,dxdr\leq C \delta^{2(\theta-\alpha)} [(-\Delta)^{\sfrac{\alpha}{2}}u]^2_{L^\infty((0,T),B^{\theta-\alpha}_{2,\infty}(\T^3))}\,.
$$
Since $\|(-\Delta)^{\sfrac{\alpha}{2}}u(t)\|_{W^{k,2}(\T^3)}\leq \|u(t)\|_{W^{\alpha+k,2}(\T^3)}$ for both $k=0, 1$, by interpolation we also get $[(-\Delta)^{\sfrac{\alpha}{2}}u(t)]_{B^{\theta-\alpha}_{2,\infty}(\T^3)}\leq \|u(t)\|_{B^{\theta}_{2,\infty}(\T^3)}$.
Thus we have achieved
$$
|E_u(t)-E_{u_\delta}(t)| \leq C\delta^{2(\theta-\alpha)}[u]^2_{L^\infty((0,T),B^\theta_{3,\infty}(\T^3))}\,.
$$
Note that the second term in the right hand side of \eqref{split_energy} is estimated by the same expression for the case $\nu=0$, thus we get
\[
|E_{u_\delta}(t)-E_{u_\delta}(s)|\leq |t-s| \bigg\| \frac{d E_{u_\delta}}{dt}\bigg\|_{L^\infty(0,T)}\leq C|t-s| \delta^{3\theta -1}[u]^3_{L^\infty((0,T),B^\theta_{3,\infty}(\T^3))}\,.
\]
Thus we have obtained
\[
|E_u(t)-E_u(s)|\leq C \big( [u]^2_{L^\infty((0,T),B^\theta_{3,\infty}(\T^3))}+[u]^3_{L^\infty((0,T),B^\theta_{3,\infty}(\T^3))} \big)\big(\delta^{2(\theta-\alpha)} +|t-s| \delta^{3\theta -1}\big)\,,
\]
from which choosing $\delta=|t-s|^\frac{1}{1-3\theta+2(\theta-\alpha)}$ we conclude the validity of \eqref{eqn:ts-energ}.

\appendix

\section{Mollification estimates}\label{A:moll}
Note that a direct consequence of the definition is that the mollification preserves the mean
\begin{equation}\label{mean_moll}
\int_{\T^3}f(x)\,dx=\int_{\T^3}f_\delta(x)\,dx \qquad \forall \delta>0\,.
\end{equation}
In the next proposition we collect some elementary estimates on these regularized functions. The simbol $\star$ is used to denote both the tensor and the scalar product.
\begin{proposition}\label{p:moll}
For any $f: \mathbb{T}^3 \rightarrow \mathbb{R}^3$, $\theta \in(0,1)$ and $N\geq 0$ we have:
\begin{align}
  \| f_\delta \star f_\delta -(f \star f)_\delta \|_{C^N} &\leq C \delta^{2\theta-N} [f]^2_{C^\theta}\,,\label{mollest4} \\
  \| \ f_\delta\|_{C^{N+\theta}} &\leq C \delta^{-N} [f]_{C^{\theta}} \, ,\label{mollest1} \\
\| \ f_\delta\|_{C^{N+1}} &\leq C \delta^{\theta-N-1} [f]_{C^{\theta}} \, ,\label{mollest3} \\
 \|f_\delta -f \|_{C^0} &\leq C \delta^{\theta} [f]_{C^{\theta}}\,.  \label{mollest2}
\end{align}
for a constant $C$ depending only on $N$ (but not on $\theta$).
\end{proposition}

We remark that similar estimates also hold for more general spaces.
\begin{proposition}
For any $1\leq p <\infty$ there exists a constant $C:=C(p)$ such that for any $f: \mathbb{T}^3 \rightarrow \mathbb{R}^3$ and for any $\theta\in (0,1)$ we have
\begin{align}
\| \nabla f_\delta\|_{L^p} &\leq C \delta^{\theta-1} \min\big\{ [f]_{W^{\theta,p}},   [f]_{B^\theta_{p,\infty}}\big\} \, ,\label{est1} \\
 \| f_\delta \star f_\delta -(f\star f)_\delta \|_{L^p} &\leq C \delta^{2\theta} \min\big\{ [f]^2_{W^{\theta,2p}} [f]^2_{B^\theta_{2p,\infty}} \big\} 
\,. \label{est4}
\end{align}
\end{proposition}
%
%
\section{Fractional Laplacian}
Let $\alpha \in \big(0,\frac{1}{2}\big)$. We racall that for a regular function $f:\T^3\rightarrow\R^3$ the operator $(-\Delta)^\alpha$ is given by (see \cite{RS2016} for a more precise discussion)
$$
(-\Delta)^\alpha f(x)= C_\alpha\sum_{k\in \mathbb{Z}^3} \int_{\T^3} \frac{f(x)-f(y)}{|x-y-2\pi k|^{3+2\alpha}}\,dy\,.
$$
If we identify the torus with the $3-$dimensional cube $[0,2\pi]^3$ and we think $f$ to be defined periodically on $\R^3$ we can rewrite the previous formula as
\begin{equation}\label{lapl_global}
(-\Delta)^\alpha f(x)= C_\alpha \int_{\R^3} \frac{f(x)-f(y)}{|x-y|^{3+2\alpha}}\,dy\,,
\end{equation}
which is a well defined integral whenever $f\in C^\theta(\T^3)$ with $\theta>2\alpha$. We have the following (see \cite[Theorem B.1]{DR2018} and \cite[Theorem 1.4 ]{RS2016} for the original version)
\begin{theorem}\label{lapla.holder}
Let $\alpha, \varepsilon>0$ and $\beta\geq0$  such that $2\alpha+\beta+\varepsilon\leq 1$, and let $f:\mathbb{T}^3 \rightarrow \mathbb{R}^3$. If  $f \in C^{2 \alpha +\beta+ \varepsilon} (\T^3)$,  then $(-\Delta)^\alpha f \in C^\beta(\T^3)$. Moreover there exists a constant $C_\varepsilon>0$ such that 
\[
\| (-\Delta)^\alpha f\|_{C^\beta(\T^3)} \leq C_\varepsilon [f]_{C^{2 \alpha +\beta+ \varepsilon}(\T^3)}\,.
\]
\end{theorem}
We also have the following commutator estimate
\begin{proposition}\label{commutator_alpha}
Let $k_1$, $k_2$, $\alpha \in \big(0,\frac{1}{2}\big)$, $\beta \in (0,2)$, $f,g:\T^3\rightarrow\R^3$ and consider the non-local operator $T^\alpha (f,g):= (-\Delta)^\alpha (f\otimes g)-(-\Delta)^\alpha f\otimes g- f\otimes(-\Delta)^\alpha g $. Assume also that $k_1+k_2 =2\alpha$. We have the following

\begin{itemize}
\item[(i)]if $\max(k_1,k_2)+\frac{\beta}{2}<2\alpha$ there exists a constant $C=C_{k_1, k_2, \alpha, \beta}>0$ such that
$$
\|T^\alpha(f,g)\|_{C^{\beta}(\T^3)}\leq C \|f\|_{C^{k_1+\sfrac{\beta}{2}}(\T^3)}\|g\|_{C^{k_2+\sfrac{\beta}{2}}(\T^3)}\,;
$$
\item[(ii)] if $\min(k_1,k_2)+\frac{\beta}{2}\geq 2\alpha$ and $\max(k_1,k_2)+\frac{\beta}{2}<1$ then for every small $\varepsilon>0$ there exists a constant $C=C_{k_1, k_2, \alpha, \beta, \varepsilon}>0$ such that
$$
\|T^\alpha(f,g)\|_{C^{\min(k_1,k_2)+\sfrac{\beta}{2}-\varepsilon}(\T^3)}\leq C \|f\|_{C^{k_1+\sfrac{\beta}{2}}(\T^3)}\|g\|_{C^{k_2+\sfrac{\beta}{2}}(\T^3)}\,.
$$
\end{itemize}
\end{proposition}
An easy consequence of the previous proposition is that, taking $f=g=u$ and $k_1=k_2=\alpha$, one gets
\begin{equation}\label{est_comm_1}
\|T^\alpha (u)\|_{C^\beta(\T^3)}\leq C_{\alpha,\beta}\|u\|^2_{C^{\alpha+\sfrac{\beta}{2}}(\T^3)}\qquad \textit{if }\, \frac{\beta}{2}<\alpha\,;
\end{equation}
\begin{equation}\label{est_comm_2}
\|T^\alpha (u)\|_{C^{\alpha +\sfrac{\beta}{2}-\varepsilon}(\T^3)}\leq C_{\alpha,\beta,\varepsilon}\|u\|^2_{C^{\alpha+\sfrac{\beta}{2}}(\T^3)}\qquad  \textit{if }\, \alpha\leq\frac{\beta}{2},\, \alpha+\frac{\beta}{2}<1\,;
\end{equation}
where we used the notation $T^\alpha(u)=T^\alpha(u,u)$.
\begin{proof}
A direct consequence of \eqref{lapl_global} is the following pointwise formula
\[
T^\alpha( f,g)(x)=C_\alpha \int_{\R^3} \frac{\big( f(x)-f(y)\big) \otimes \big( g(y)-g(x)\big)}{|x-y|^{3+2\alpha}}\,dy\,.
\]

%
The estimate $\|T^\alpha f\|_{C^0}\leq C \|f\|_{C^{k_1+\sfrac{\beta}{2}}}\|g\|_{C^{k_2+\sfrac{\beta}{2}}}$ it is easy and is left to the reader.\\

We fix $x_1,x_2 \in \R^3$ and we define $\overline{x}:=\frac{x_1+x_2}{2}$ and $\lambda=|x_1-x_2|$. For simplicity we also define the tensor $\varphi(x,y):=\big(f(x)-f(y)\big)\otimes\big(g(y)-g(x)\big)$. We now split 
\begin{align*}
\frac{T^\alpha (f,g)(x_1)-T^\alpha (f,g)(x_2)}{C_\alpha}&= \int_{B_\lambda(\overline{x})}\frac{\varphi(x_1,y)}{|x_1-y|^{3+2\alpha}}\,dy+\int_{B_\lambda(\overline{x})}\frac{\varphi(x_2,y)}{|x_2-y|^{3+2\alpha}}\,dy
+\int_{B^c_\lambda(\overline{x})}\frac{\varphi(x_1,y)-\varphi(x_2,y)}{|x_1-y|^{3+2\alpha}}\,dy\\&+\int_{B^c_\lambda(\overline{x})}\Big( \frac{1}{|x_1-y|^{3+2\alpha}}- \frac{1}{|x_2-y|^{3+2\alpha}}\Big) \varphi(x_2,y)\,dy=I+II+III+IV\,.
\end{align*}
The first two integrals can be estimated as
\begin{equation}\label{primidue}
|I|,|II|\leq [f]_{C^{k_1+\sfrac{\beta}{2}}}[g]_{C^{k_2+\sfrac{\beta}{2}}}\int_{B_\lambda(\overline{x})}\Big(\frac{1}{|x_1-y|^{3-\beta}}+ \frac{1}{|x_2-y|^{3-\beta}} \Big)\,dy\leq C \lambda^\beta [f]_{C^{k_1+\sfrac{\beta}{2}}}[g]_{C^{k_2+\sfrac{\beta}{2}}}\,.
\end{equation}
Now we note that the difference $\varphi(x_1,y)-\varphi(x_2,y)$ can be rewritten as
$$
\varphi(x_1,y)-\varphi(x_2,y)=\big(f(x_1)-f(x_2)\big)\otimes \big(g(y)-g(x_2)\big)+\big(f(y)-f(x_1)\big)\otimes \big(g(x_1)-g(x_2)\big)\,.
$$
Thus, assuming $\max(k_1,k_2)+\frac{\beta}{2}<2\alpha$, we estimate
\begin{align}\label{terzo_uno}
|III|&\leq C [f]_{C^{k_1+\sfrac{\beta}{2}}}[g]_{C^{k_2+\sfrac{\beta}{2}}}\Bigg(\int_{B^c_\lambda(\overline{x})}\frac{\lambda^{k_1 +\frac{\beta}{2}}}{|\overline x-y|^{3+2\alpha -k_2-\frac{\beta}{2}}}\,dy+\int_{B^c_\lambda(\overline{x})}\frac{\lambda^{k_2 +\frac{\beta}{2}}}{|\overline x-y|^{3+2\alpha -k_1-\frac{\beta}{2}}}\,dy\Bigg) \nonumber \\
&\leq C\lambda^{\beta} [f]_{C^{k_1+\sfrac{\beta}{2}}}[g]_{C^{k_2+\sfrac{\beta}{2}}}
\end{align}
while in the case  $\min(k_1,k_2)+\frac{\beta}{2}\geq 2\alpha$, for every small $\varepsilon>0$ we estimate
\begin{align}\label{terzo_due}
|III|&\leq C \Bigg( [f]_{C^{k_1+\sfrac{\beta}{2}}}[g]_{C^{2\alpha-\varepsilon}}\int_{B^c_\lambda(\overline{x})}\frac{\lambda^{k_1 +\frac{\beta}{2}}}{|\overline x-y|^{3 +\varepsilon}}\,dy +[f]_{C^{2\alpha-\varepsilon}}[g]_{C^{k_2 +\frac{\beta}{2}}}\int_{B^c_\lambda(\overline{x})}\frac{\lambda^{k_2+\frac{\beta}{2}}}{|\overline x-y|^{3 +\varepsilon}}\,dy\Bigg) \nonumber\\
&\leq C \Big(\lambda^{k_1+\frac{\beta}{2}-\varepsilon} +\lambda^{k_2+\frac{\beta}{2}-\varepsilon} \Big)\|f\|_{C^{k_1+\sfrac{\beta}{2}}}\|g\|_{C^{k_2+\sfrac{\beta}{2}}}\,,
\end{align}
where we have also used that $|x_1-y|,|x_2-y|\gtrsim |\overline{x}-y|$ for every $y\in B^c_\lambda(\overline{x})$.\\
 We notice that for every $y\in B^c_\lambda(\overline{x})$ we have
$$
\Big|\frac{1}{|x_1-y|^{3+2\alpha}}- \frac{1}{|x_2-y|^{3+2\alpha}}\Big|=\Big|\int_0^1\frac{d}{dt}\frac{1}{|tx_1+(1-t)x_2-y |^{3+2\alpha}}\,dt\Big|\lesssim \lambda \frac{1}{|\overline{x}-y|^{4+2\alpha}}
$$
from which, in the case $\max(k_1,k_2)+\frac{\beta}{2}<2\alpha$, we get (notice that in this case $ \beta<4\alpha-2\max(k_1,k_2)<2\alpha<1$)
\begin{equation}\label{quarto}
|IV|\leq C \lambda [f]_{C^{k_1+\sfrac{\beta}{2}}}[g]_{C^{k_2+\sfrac{\beta}{2}}}  \int_{B_\lambda^c(\overline{x})}\frac{1}{|\overline x -y|^{4-\beta}}\,dy\leq C \lambda^\beta [f]_{C^{k_1+\sfrac{\beta}{2}}}[g]_{C^{k_2+\sfrac{\beta}{2}}} \,,
\end{equation}
while, if  $\min(k_1,k_2)+\frac{\beta}{2}\geq 2\alpha$ and $\max(k_1,k_2)+\frac{\beta}{2}<1$ we have
\begin{equation}\label{quarto_due}
|IV|\leq C \lambda [f]_{C^{k_1+\sfrac{\beta}{2}}}[g]_{C^{2\alpha}}\int_{B_\lambda^c(\overline{x})}\frac{1}{|\overline x -y|^{4-k_1-\frac{\beta}{2}}}\,dy\leq C\lambda^{k_1+\frac{\beta}{2}}\|f\|_{C^{k_1+\sfrac{\beta}{2}}}\|g\|_{C^{k_2+\sfrac{\beta}{2}}} \,.
\end{equation}
We conclude the proof combining \eqref{primidue}-\eqref{quarto_due}.
\end{proof}


\begin{thebibliography}{BDLISJ15}

\bibitem[DS13]{DS2013}
C. De Lellis, L. Székelyhidi Jr.
\newblock Dissipative continuous Euler flows.
\newblock {\em Inventiones Mathematicae} 193, Issue 2 (2013), Page 377-407.

\bibitem[DLS12]{DLS2012}
C.~De~Lellis and  L.~Sz\'ekelyhidi, Jr.
\newblock Dissipative Euler Flows and Onsager's Conjecture.
\newblock {\em Journal of the European Mathematical Society}  16(7),2017.


\bibitem[Ey94]{Eyink94}
G.~ L.~Eyink.
\newblock Energy dissipation without viscosity in ideal hydrodynamics I. Fourier analysis and local energy transfer.
\newblock {\em Physica D: Nonlinear Phenomena}, 78:222--240, 1994.

\bibitem[RS16]{RS2016}
L. Roncal and P. R. Stinga
\newblock Fractional Laplacian on the torus.
\newblock {\em Comm.  Contemp. Math.} 18 No. 03, 2016.



\bibitem[DR18]{DR2018}
 L. De Rosa.
\newblock Infinitely many Leray-Hopf solutions for the fractional Navier-Stokes equations.
\newblock {\em preprint}  	arXiv:1801.10235 , 2018.

\bibitem[Ons49]{Ons1949}
L.~Onsager.
\newblock {S}tatistical hydrodynamics.
\newblock {\em Il {N}uovo {C}imento (1943-1954)}, 6:279--287, 1949.


\bibitem[CDD17]{CDD2017}
M. Colombo, C. De Lellis and  L.~De~Rosa.
\newblock Ill-Posedness of Leray Solutions for the Hypodissipative Navier–Stokes Equations.
\newblock {\em Comm.  Math. Phys.}  1--30, 2017.

\bibitem[K41]{K1941}
 N. Kolmogorov. 
 \newblock The local structure of turbulence in an incompressible viscous fluid.
\newblock C.R. (Doklady) Acad. Sci. URSS (N.S.), 30:301–305, 1941.

\bibitem[Is13]{Is2013}
 P. Isett.
\newblock Regularity in time along the coarse scale flow for the incompressible Euler equations.
\newblock {\em preprint}  arXiv:1307.0565, 2013.

\bibitem[Is16]{Is2016}
P.~Isett
\newblock A proof of Onsager's cnjecture.
\newblock {\em Ann. of Math.}  accepted.


\bibitem[CET94]{CoETi1994}
P.~Constantin, W.~E, and E.S. Titi.
\newblock Onsager's conjecture on the energy conservation for solutions of
  {E}uler's equation.
\newblock {\em Comm. Math. Phys.}, 165(1):207--209, 1994.

\bibitem[RS16]{RS2016}
L.~Roncal and P.~R.~Stinga.
\newblock Fractional Laplacian on the torus.
\newblock {\em Commun. Contemp. Math},   	18(3):1550033, 26, 2016.


\bibitem[BDLSV17]{BDLSV2017}
T.~Buckmaster, C.~De~Lellis, L.~Sz\'ekelyhidi, Jr. and V.~Vicol.
\newblock Onsager's conjecture for admissible weak solutions.
\newblock {\em Preprint},  arXiv:1701.08678v1 ,2017.

\bibitem[BDLIS15]{DLS2015}
T.~Buckmaster, C.~De~Lellis, P.~Isett and L.~Székelyhidi Jr.
\newblock Anomalous dissipation for 1/5-Hoelder Euler flows.
\newblock {\em Ann. of Math.}  182(2): 127--172 ,2015.





\bibitem[BDS16]{BDS2016}
T. Buckmaster, C. De Lellis, L. Székelyhidi Jr.
\newblock Dissipative Euler flows with Onsager-critical spatial regularity.
\newblock {\em Comm. Pure Appl. Math.} 69 (2016), no. 9, 1613–1670.

\bibitem[BV17]{BV2017}
T.~Buckmaster and V.~Vicol.
\newblock Nonuniqueness of weak solutions to the Navier-Stokes equation.
\newblock {\em 	arXiv:1709.10033 [math.AP]}, 2017.


\end{thebibliography}
\end{document}